\documentclass[10pt]{article}
\usepackage{amssymb}
\usepackage{amsmath}
\usepackage{amsfonts}
\usepackage{amsthm}

\usepackage[english,russian]{babel}

\newtheorem{theorem}{Теорема}
\newtheorem{lemma}{Лемма}
\newtheorem{proposition}{Утверждение}

\title{Минимальные носители некоторых графов Хэмминга}
\author{Alexandr Valyuzhenich}
\begin{document}
\maketitle

\begin{abstract}
В работе найдены минимальные носители собственных функций графов Хэмминга $H(n,q)$ с собственным значением $n(q-1)-q$. Кроме того, получена полная характеризация собственных функций, на которых достигается минимальное значение носителя.

\end{abstract}
\section{Введение}
{\em Расстоянием Хэмминга} $d(x,y)$ между словами $x,y\in{\{0,1,\ldots,q-1\}^{n}}$ называется число позиций, в которых $x$ и $y$ различны.
{\em Графом Хэмминга} называется граф, вершины которого --- это все слова длины $n$ над алфавитом $\{0,1,\ldots,q-1\}$, а ребрами графа соединяются вершины на расстоянии Хэмминга $1$. Обозначим граф Хэмминга через $H(n,q)$.
Хорошо известно, что множество собственных значений матрицы смежности графа $H(n,q)$ --- это $\{\lambda_{m}=n(q-1)-qm|m=0,1,\ldots,n\}$.
Функция $f:H(n,q)\longrightarrow{\mathbb{R}}$ называется {\em собственной функцией} графа $H(n,q)$, отвечающей собственному значению $\lambda$, если
$Af=\lambda f$, где $A$ --- матрица смежности $H(n,q)$.
Пусть $f:H(n,q)\longrightarrow{\mathbb{R}}$. Множество $S(f)=\{x\in{H(n,q)}|f(x)\neq0)\}$ называется \emph{носителем функции} $f$.
Для носителя собственной функции известна следующая нижняя оценка:
\begin{theorem}[{\cite{vor}}]
Пусть  $f:H(n,q)\longrightarrow{\mathbb{R}}$ --- собственная функция, отвечающая собственному значению $\lambda_m$ и $f\not\equiv0$. Тогда $|S(f)|\geq{2^m(q-2)^{n-m}}$ для $\frac{mq^2}{2n(q-1)}>2$ и $|S(f)|\geq{q^n(\frac{1}{q-1})^{m/2}(\frac{m}{n-m})^{m/2}(1-\frac{m}{n})^{n/2}}$ для $\frac{mq^2}{2n(q-1)}\leq2$.
\end{theorem}

Из результатов работы \cite{pot} следует, что для мощности носителя собственной функции $f:H(n,q)\longrightarrow{\{-1,0,1\}}$, отвечающей собственному значению $\lambda=q(n-m)-n$, выполнена нижняя оценка $|S(f)|\geq{2^m}$.
\section{Лемма о редукции}
Множество вершин $x$ графа $H(n,q)$, у которых $i$--я координата равна $k$, обозначим через $T_{k}(i,n)$.

Пусть $t=(t_1,t_2,\ldots,t_n)$ --- произвольная вершина графа $H(n,q)$. Рассмотрим вектора $x=(t_1,\ldots,t_{i-1},k,t_{i},\ldots,t_n)$ и $y=(t_1,\ldots,t_{i-1},m,t_{i},\ldots,t_n)$ длины $n+1$. Заметим, что $x\in{T_{k}(i,n+1)}$ и $y\in{T_{m}(i,n+1)}$, и вектор $t$ получается из вершин $x$ и $y$ удалением $i$--ой координаты.
Определим функцию $g_{i,k,m}:H(n,q)\longrightarrow{\mathbb{R}}$ по правилу $g_{i,k,m}(t)=f(x)-f(y)$.
\begin{lemma}\label{reduction}
Пусть  $f:H(n+1,q)\longrightarrow{\mathbb{R}}$ --- собственная функция, отвечающая собственному значению $\lambda$. Тогда $g_{i,k,m}(t)$ --- собственная функция в графе $H(n,q)$, отвечающая собственному значению $\lambda+1$.
\end{lemma}
\begin{proof}
Пусть $t=(t_1,t_2,\ldots,t_n)$ --- произвольная вершина графа $H(n,q)$. Рассмотрим вершины $x=(t_1,\ldots,t_{i-1},k,t_{i},\ldots,t_n)$ и $y=(t_1,\ldots,t_{i-1},m,t_{i},\ldots,t_n)$ графа $H(n+1,q)$. Заметим, что $x\in{T_{k}(i,n+1)}$ и $y\in{T_{m}(i,n+1)}$, причем $x$ и $y$ смежны. Пусть $x_1,x_2,\ldots,x_s$ --- соседи $x$ среди вершин из $T_{k}(i,n+1)$.  Заметим, что $y$ среди вершин из $T_{m}(i,n+1)$ имеет соседей $y_1,y_2,\ldots,y_s$, причем $x_t$ и $y_t$ отличаются только в $i$--ой координате. Вектор длины $n$, полученный удалением из вершин $x_t$ и $y_t$ $i$--ой координаты, обозначим через $z_t$. Пусть $p_t$ --- сосед вершины $x$, лежащий в $T_{t}(i,n+1)$, и $P=\{p_0,p_1,\ldots,p_{q-1}\}$. Тогда $N(x)=\{x_1,x_2,\ldots,x_s\}\cup\{P\setminus{x}\}$. Так как $f$ --- собственная функция, то имеем
\begin{equation}\label{eq 1}
\lambda f(x)=\sum_{i=1}^{s}f(x_i)+\sum_{i=0}^{q-1}f(p_i)-f(x).
\end{equation}

Аналогичным образом имеем $N(y)=\{y_1,y_2,\ldots,y_s\}\cup\{P\setminus{y}\}$. Так как $f$ --- собственная функция, то имеем
\begin{equation}\label{eq 2}
\lambda f(y)=\sum_{i=1}^{s}f(y_i)+\sum_{i=0}^{q-1}f(p_i)-f(y).
\end{equation}
Вычитая из соотношения~\ref{eq 1} соотношение~\ref{eq 2}, получаем

$(\lambda+1)(f(x)-f(y))=\sum_{i=1}^{s}(f(x_i)-f(y_i))$. Тогда $(\lambda+1)g_{i,k,m}(t)=\sum_{i=1}^{s}g_{i,k,m}(z_i)$ для произвольной вершины $t$. Так как в графе $H(n,q)$ вершина $t$ имеет соседей $z_1,z_2,\ldots,z_s$, то $g_{i,k,m}(t)$ --- собственная функция в графе $H(n,q)$.

\end{proof}

Фунцию $f:H(n+1,q)\longrightarrow{\mathbb{R}}$ будем называть {\em аддитивной}, если для любых допустимых $i,k,m$ функция $g_{i,k,m}$ является константой.

\begin{lemma}\label{L2}
Пусть  $f:H(n+1,q)\longrightarrow{\mathbb{R}}$ --- собственная функция, отвечающая собственному значению $\lambda_1$. Тогда $f$ --- аддитивная функция.
\end{lemma}
\begin{proof}
Достаточно доказать, что для всех допустимых $i$, $j$ и $p$ функция $g_{p,i,j}$ является константой. По лемме~\ref{reduction} функция $g_{p,i,j}$ является собственной функцией графа $H(n,q)$, отвечающей собственному значению $\lambda_0$. Так как $\lambda_0$ имеет кратность $1$, то любая собственная функция графа $H(n,q)$, отвечающая собственному значению $\lambda_0$, является константой. Лемма доказана.

\end{proof}

\section{Основная теорема}
Множество вершин $x$ графа $H(n,q)$, у которых $i$--я координата равна $k$, обозначим через $T_{k}(i,n)$.
\begin{lemma}\label{L3}
Пусть  $f:H(2,q)\longrightarrow{\mathbb{R}}$ --- аддитивная функция, $|S(f)|\leq{2(q-1)}$. Тогда выполнено одно из следующих условий:
\begin{enumerate}
\item $f\equiv0$.

\item $f(x)=c$ при $x\in{T_{k}(i,2)}$ и $f(x)=0$ в остальных случаях, где $c\neq{0}$ --- некоторая константа.

\item
$
f(x)=\begin{cases}
c,&\text{при $x\in{T_{k}(i,2)}\setminus{T_{m}(j,2)}$;}\\
-c,&\text{при $x\in{T_{m}(j,2)}\setminus{T_{k}(i,2)}$;}\\
0,&\text{иначе.}
\end{cases}
,$ где  $c\neq{0}$ --- некоторая константа и $i\neq{j}$.
\end{enumerate}
\end{lemma}
\begin{proof}
Так как $|S(f)|\leq{2(q-1)}$, то $f(x)=0$ для некоторого $x$. Без ограничения общности, пусть $f(x_0)=0$, где $x_0=(0,0)$. Значения $f$ при $x=(0,j)$ и $x=(i,0)$ обозначим через $a_j$ и $b_i$ соответственно. Пусть среди чисел $a_1,\ldots,a_{q-1}$ ровно $k$ ненулевых, среди $b_1,\ldots,b_{q-1}$ ровно $s$ ненулевых. Множество вершин $x=(i,j)$ таких, что $j>0$, $f(y)\neq{0}$ для $y=(i,0)$ и $f(z)={0}$ для $z=(0,j)$, обозначим через $B$. Множество вершин $x=(i,j)$ таких, что $i>0$, $f(y)={0}$ для $y=(i,0)$ и $f(z)\neq{0}$ для $z=(0,j)$, обозначим через $C$.

Пусть $x=(i,j)$, $y=(0,j)$ и $z=(i,0)$.
Так $f$ --- аддитивная функция, то $f(x)-f(y)=f(z)-f(x_0)$. Отсюда $f(x)=a_j+b_i$. Пусть $x=(i,j)\in{B}$. Тогда $f(x)=a_j+b_i=b_i\neq{0}$. Аналогично имеем, что $f(x)=a_j\neq{0}$ для $x=(i,j)\in{C}$. Значит $|S(f)|\geq{k+s+|B|+|C|}$. Поэтому $|S(f)|\geq{k+s+(q-k-1)s+(q-s-1)k}$. Рассмотрим два случая.

\textbf{Случай 1.} $k\geq{1}$ и $s\geq{1}$. Имеем $|S(f)|\geq{k+s+(q-k-1)s+(q-s-1)k}$. Поэтому $|S(f)|\geq{2(q-1)+(q-k-1)(s-1)+(q-s-1)(k-1)}$. С другой стороны,  $|S(f)|\leq{2(q-1)}$. Поэтому либо $k=s=q-1$, либо $k=s=1$. Пусть $k=s=q-1$. Тогда $f(x)=a_j+b_i=0$ для $i>0$ и $j>0$. Поэтому все $a_i$ равны $c$, все $b_j$ равны $-c$, где $c$ --- некоторая константа. Поэтому в этом случае

$f(x)=\begin{cases}
c,&\text{при $x\in{T_{0}(2,2)}\setminus{T_{0}(1,2)}$;}\\
-c,&\text{при $x\in{T_{0}(1,2)}\setminus{T_{0}(2,2)}$;}\\
0,&\text{иначе.}
\end{cases}
$.

В случае $k=s=1$ аналогичным образом получаем, что

$f(x)=\begin{cases}
c,&\text{при $x\in{T_{1}(2,2)}\setminus{T_{1}(1,2)}$;}\\
-c,&\text{при $x\in{T_{1}(1,2)}\setminus{T_{1}(2,2)}$;}\\
0,&\text{иначе.}
\end{cases}
$.

\textbf{Случай 2.} $k=0$ или $s=0$. Без ограничения общности, пусть $k=0$. В этом случае $|B|=(q-1)s$. Так как $|S(f)|\leq{2(q-1)}$, то $s<2$ или $s=0$. Если $s=0$, то $f\equiv0$. Осталось рассмотреть случай $s=1$. В этом случае $S(f)=B$, $f(x)=b_t$ для некоторого $t$ при $x\in{B}$ и $f(x)=0$ для остальных $x$.

\end{proof}

Множество вершин $x$ графа $H(n,q)$, у которых $i$--я координата равна $k$, а $j$--я равна $m$, обозначим через $T_{k,m}(i,j,n)$ ($i\neq{j}$).
\begin{theorem}\label{T2}
Пусть  $f:H(n,q)\longrightarrow{\mathbb{R}}$ --- аддитивная функция, $|S(f)|\leq{2(q-1)q^{n-2}}$, $q>2$ и $n>1$. Тогда выполнено одно из следующих условий:
\begin{enumerate}
\item $f\equiv0$.

\item $f(x)=c$ при $x\in{T_{k}(i,n)}$ и $f(x)=0$ в остальных случаях, где $c\neq{0}$ --- некоторая константа.

\item
$
f(x)=\begin{cases}
c,&\text{при $x\in{T_{k}(i,n)}\setminus{T_{m}(j,n)}$;}\\
-c,&\text{при $x\in{T_{m}(j,n)}\setminus{T_{k}(i,n)}$;}\\
0,&\text{иначе.}
\end{cases}
,$ где  $c\neq{0}$ --- некоторая константа и $i\neq{j}$.
\end{enumerate}
\end{theorem}
\begin{proof}
Докажем теорему индукцией по $n$. База для $n=2$ доказана в лемме~\ref{L3}.

Докажем переход от $n$ к $n+1$. Пусть  $f:H(n+1,q)\longrightarrow{\mathbb{R}}$ --- аддитивная функция. Функцию $f|_{T_{i}(n+1,n+1)}$ обозначим через $f_i$. Так как $f$ --- аддитивная функция, то для всех $i$  имеем $g_{n+1,i,0}\equiv{c_i}$ для некоторой константы $c_i$. Пусть среди чисел $c_0,c_1,\ldots,c_{q-1}$ ровно $s$ ненулевых.  Так как $|S(f)|\leq{2(q-1)q^{n-1}}$, то для некоторого $i$ имеем $|S(f_i)|\leq{2(q-1)q^{n-2}}$. Без ограничения общности, пусть $i=0$. Таким образом, $|S(f_0)|\leq{2(q-1)q^{n-2}}$ и $f_0$ --- аддитивная функция. Тогда по предположению индукции для $f_0$ возможны три варианта.

\textbf{Случай 1.} В этом случае $f_0\equiv0$. Тогда $|S(f)|=s|S(f_0)|=sq^n$. Так как $|S(f)|\leq{2(q-1)q^{n-1}}$, то $s<2$. Если $s=0$, то $f_i\equiv0$ для всех $i$. Отсюда $f\equiv0$. Пусть $s=1$ и $c_k\neq{0}$. Тогда $f(x)=c_k$ при $x\in{T_{k}(n+1,n+1)}$ и $f(x)=0$ в остальных случаях.

\textbf{Случай 2.} В этом случае $f_0(x)=c$ при $x\in{T_{0,k}(n+1,i,n+1)}$ и $f(x)=0$ в остальных случаях, где $c\neq{0}$ --- некоторая константа.  Заметим, что для $c_i\neq{0}$ выполнено неравенство $|S(f_i)|\geq{|T_{i}(n+1,n+1)|-|S(f_0)|}$, причем равенство выполняется только для $c_i=-c$. Отсюда $|S(f_i)|\geq{(q-1)q^{n-1}}$, причем равенство выполняется только для $c_i=-c$.

Тогда $|S(f)|\geq{|S(f_0)|+s(q-1)q^{n-1}}$. Так как $|S(f)|\leq{2(q-1)q^{n-1}}$, то $s<2$. Если $s=0$, то $f_m(x)=c$ при $x\in{T_{m,k}(n+1,i,n+1)}$ и $f(x)=0$ в остальных случаях для всех $m$. Отсюда $f(x)=c$ при $x\in{T_{k}(i,n+1)}$ и $f(x)=0$ в остальных случаях. Пусть $s=1$ и $c_m\neq{0}$. Тогда $|S(f_i)|=|S(f_0)|=q^{n-1}$ при $i\neq{m}$.  Так как $|S(f)|\leq{2(q-1)q^{n-1}}$, то $|S(f_m)|\leq{(q-1)q^{n-1}}$. Поэтому $|S(f_m)|=(q-1)q^{n-1}$ и $c_m=-c$.

Отсюда
$f(x)=\begin{cases}
c,&\text{при $x\in{T_{k}(i,n+1)}\setminus{T_{m}(n+1,n+1)}$;}\\
-c,&\text{при $x\in{T_{m}(n+1,n+1)}\setminus{T_{k}(i,n+1)}$;}\\
0,&\text{иначе.}
\end{cases}
$

\textbf{Случай 3.} В этом случае $|S(f_0)|=2(q-1)q^{n-2}$.  Так как для $c_i\neq{0}$ выполнено неравенство $|S(f_i)|\geq{|T_{i}(n+1,n+1)|-|S(f_0)|}$, то $|S(f_i)|\geq{q^n-2(q-1)q^{n-2}}$ для $c_i\neq{0}$. Поэтому $|S(f_i)|>{2(q-1)q^{n-2}}$ при $c_i\neq{0}$ и $q\geq{3}$. Таким образом,  $|S(f)|\geq{q(2(q-1)q^{n-2})}$, причем равенство возможно только в случае, когда все $c_i$ равны нулю. Следовательно, $c_i=0$ для любого $i$.

Поэтому для любого $s$ имеем

$f_s(x)=\begin{cases}
c,&\text{при $x\in{T_{s,k}(n+1,i,n)}\setminus{T_{s,m}(n+1,j,n)}$;}\\
-c,&\text{при $x\in{T_{s,m}(n+1,j,n)}\setminus{T_{s,k}(n+1,i,n)}$;}\\
0,&\text{иначе.}
\end{cases}
$.

Следовательно, $f(x)=\begin{cases}
c,&\text{при $x\in{T_{k}(i,n+1)}\setminus{T_{m}(j,n+1)}$;}\\
-c,&\text{при $x\in{T_{m}(j,n+1)}\setminus{T_{k}(i,n+1)}$;}\\
0,&\text{иначе.}
\end{cases}
$.

\end{proof}

\begin{proposition}\label{Pr1}
Пусть
$f(x)=\begin{cases}
c,&\text{при $x\in{T_{k}(i,n)}\setminus{T_{m}(j,n)}$;}\\
-c,&\text{при $x\in{T_{m}(j,n)}\setminus{T_{k}(i,n)}$;}\\
0,&\text{иначе.}
\end{cases}
$

где  $c\neq{0}$ --- некоторая константа, $i,j,k,m$ --- некоторые числа.
Тогда  $f$ --- собственная функция, отвечающая собственному значению $\lambda_1$.
\end{proposition}
\begin{proof}

Пусть $x$ --- произвольная вершина $H(n,q)$, для которой $f(x)=c$. Тогда $x$ имеет $(n-1)(q-1)-1$ соседей в $T_{k}(i,n)\setminus{T_{m}(j,n)}$ и не имеет соседей в  $T_{m}(j,n)\setminus T_{k}(i,n)$ . Тогда сумма значений функции $f$ по всем соседям $x$ равна $c(n-1)(q-1)-c$, то есть равна $\lambda_1f(x)$. Случай $f(x)=-c$ рассматривается аналогично.

Рассмотрим вершину $x$, для которой $f(x)=0$. Эта вершина имеет одного соседа из $T_{k}(i,n)
\setminus T_{m}(j,n)$ и одного соседа из $T_{m}(j,n)\setminus T_{k}(i,n)$. Тогда сумма значений функции $f$ по всем соседям $x$ равна $0$, то есть равна $\lambda_1f(x)$.

\end{proof}

Теперь докажем основную теорему данной работы.
\begin{theorem}
Пусть  $f:H(n,q)\longrightarrow{\mathbb{R}}$ --- собственная функция, отвечающая собственному значению $\lambda_1$ и $q>2$. Тогда $|S(f)|\geq{2(q-1)q^{n-2}}$. Более того, если $|S(f)|={2(q-1)q^{n-2}}$, то
$
f(x)=\begin{cases}
c,&\text{при $x\in{T_{k}(i,n)}\setminus{T_{m}(j,n)}$;}\\
-c,&\text{при $x\in{T_{m}(j,n)}\setminus{T_{k}(i,n)}$;}\\
0,&\text{иначе.}
\end{cases}
,$

где  $c\neq{0}$ --- некоторая константа, $i,j,k,m$ --- некоторые числа, причем $i\neq{j}$.

\end{theorem}
\begin{proof}
Пусть $f$ --- собственная функция с минимальным носителем. По утверждению~\ref{Pr1} $|S(f)|\leq{2(q-1)q^{n-2}}$.
Так как $f$ --- собственная функция, то по лемме~\ref{L2} $f$ --- аддитивная функция. Так как $|S(f)|\leq{2(q-1)q^{n-2}}$, то по теореме~\ref{T2} для $f$ возможны три варианта. Учитывая, что $f$ --- собственная функция, получаем, что возможен только третий вариант, что и доказывает теорему.

\end{proof}

\section{Благодарности}
Автор выражает глубокую благодарность Д. С. Кротову, И. Ю. Могильных, К. В. Воробьеву и В. Н. Потапову за полезные замечания и обсуждения.


\begin{thebibliography}{5}
\bibitem{vor}
К. В Воробьев, Д. С. Кротов. Оценки мощности минимального 1-совершенного битрейда в графе Хэмминга, Дискретн. анализ и исслед. опер., 2014,
Т. 21, вып. 6, С. 3--10.
\bibitem{pot}
V. N. Potapov. On perfect 2-colorings of the q-ary n-cube, Discrete Math., 2012, V. 312, N. 8, 1269-1272.


\end{thebibliography}
\end{document}